\documentclass[11pt]{article}

\usepackage{amssymb,amsmath}
\input{epsf.sty}
\newcommand{\barF}{\overline{F}}
\newcommand{\barW}{\overline{W}}

\newtheorem{remark}{Remark}

\newtheorem{example}{Example} 
\newtheorem{proposition}{Proposition} 
\newtheorem{corollary}{Corollary}

\newenvironment{proof}{\begin{trivlist}
\item[\hspace{\labelsep}{\bf\noindent Proof. }]}
{$\hfill\Box$\end{trivlist}}

\title{\huge \bf On dynamic mutual information \\ for bivariate lifetimes 
}

\author{
{\sc Jafar Ahmadi}\footnote{Department of Statistics,
 Ordered and Spatial Data Center of Excellence,
 Ferdowsi University of Mashhad, P.O. Box 1159,
 Mashhad 91775, Iran, E-mail: ahmadi-j@um.ac.ir}
 , \ 
{\sc Antonio Di Crescenzo}\footnote{Dipartimento di Matematica,
 Universit\`a degli Studi di Salerno,
 Via Giovanni Paolo II, 132, 84084 Fisciano (SA), Italy,
 E-mail: adicrescenzo@unisa.it} 
 , \ 
{\sc Maria Longobardi}\footnote{Dipartimento di Matematica e Applicazioni,
 Universit\`a di Napoli Federico II,
 Via Cintia, 80126 Napoli, Italy,
 E-mail: maria.longobardi@unina.it}
} 

\date{\normalsize 
\bf First published in {\em Advances in Applied Probability}, Vol.\ 47, No.\ 4, p.\ 1157--1174\\
\copyright\ 2015 by the Applied Probability Trust
}

\begin{document}

\maketitle

\begin{abstract}
We consider dynamic versions of the mutual information of lifetime distributions, with a focus on past
lifetimes, residual lifetimes and mixed lifetimes evaluated at different instants. This allows us to study
multicomponent systems, by measuring the dependence in conditional lifetimes of two components
having possibly different ages. We provide some bounds, and investigate the mutual information of
residual lifetimes within the time-transformed exponential model (under both the assumptions of
unbounded and truncated lifetimes). Moreover, with reference to the order statistics of a random
sample, we evaluate explicitly the mutual information between the minimum and the maximum,
conditional on inspection at different times, and show that it is distribution-free in a special case.
Finally, we develop a copula-based approach aiming to express the dynamic mutual information
for past and residual bivariate lifetimes in an alternative way. 


\medskip\noindent
{\em Maths Subject Classification:}  Primary {94A17}; Secondary {62N05; 60E99}
\end{abstract}

\section{Introduction and background}
Information measures are largely used in applied contexts in order to describe useful notions related to
stochastic models. The problem of measuring the information content in a dynamic setting arises in 
various fields, such as survival analysis, reliability, and mathematical finance, for example. 
Significant results in this area have been provided in Ebrahimi {\em et al.}\ \cite{EKS2007}, 
where the focus was directed on the joint, marginal, and conditional entropies, and the mutual 
information for residual life distributions in multivariate settings.
In this paper we provide some further insight on the dynamic mutual information, with reference to
past lifetimes, residual lifetimes, and mixed lifetimes evaluated at different ages.
\par
In probability theory the mutual information of two random variables is a measure of their 
mutual dependence, and can be evaluated by means of the joint and marginal distributions. 
See Ebrahimi {\em et al.}\ \cite{EHSV2010} for a  contribution dealing with the mutual
information of certain classes of bivariate distributions, and Arellano-Valle {\em et al.}\ \cite{AVCRG2012}
for a recent investigation on the mutual information of multivariate skew-elliptical distributions.
Other kinds of multivariate information measures have been investigated by
Ebrahimi {\em et al}.\ \cite{ESZ2008}. We also mention that a nonparametric and binless
estimator for the mutual information of a $d$-dimensional random vector has been proposed
recently by Giraudo {\em et al.}\ \cite{GSS2013}.
\par
In view of suitable applications in contexts of reliability theory, in this paper we consider both 
the dynamic extensions of the mutual information and the related entropies. 
Specifically, we aim to study the applications of mutual information to the cases of past, 
residual and mixed distributions. In Section 2 we briefly recall the relevant mathematical 
concepts related to mutual information and entropy, and then introduce the bivariate 
distributions describing two lifetimes conditional on possibly different inspection times. 
In Section 3 we introduce the dynamic mutual information of past lifetimes. We obtain a 
bound for such a measure, which is suitable to describe stochastic models whose uncertainty 
is related to the past. Section 4 is concerning the mutual information of residual lifetimes. 
We provide a bound and a connection between past  and residual mutual information. 
We also investigate such a measure within the time-transformed exponential model 
(both in the classical case of unbounded lifetimes and in the new setting involving truncated 
lifetimes). In Section 5 we study the dynamic mutual information for mixed lifetimes and apply 
it to ordered data. With reference to the order statistics $X_{i:n}$, $i=1,2,\ldots, n$, we evaluate 
explicitly the mutual information between the minimum and the maximum $(X_{1:n}, X_{n:n})$ 
conditional on $(X_{1:n}\leq s, X_{n:n}>t)$ for $s<t$ and show that it is distribution-free in 
a special case. This also allows us to describe the information content in $n$-component 
systems inspected at two different times. Finally, in Section 6 we discuss a copula-based 
approach, which allows us to express the dynamic mutual information for past and residual 
bivariate lifetimes in terms of copula and survival copula, respectively.
\par
Throughout the paper we denote by $[Z|B]$ a random variable or a random vector whose
distribution is identical to the conditional distribution of $Z$ given $B$.  Moreover, primes 
denote derivatives. 
\section{Preliminaries}
Let $(X,Y)$ be a random vector, where $X$ and $Y$ are nonnegative absolutely continuous
random variables. We denote by $f(x,y)$ the joint probability density function (PDF) 
of $(X,Y)$, and by $f_X(x)$ and $f_Y(y)$ the marginal densities of $X$ and $Y$, respectively. 
It is well known that the mutual information of $X$ and $Y$ is defined as
\begin{equation}
 M_{X,Y}=\int_0^{+\infty}{\rm d}x\int_0^{+\infty}f(x,y)\,
 \log \frac{f(x,y)}{f_X(x)\,f_Y(y)}\,{\rm d}y,
 \label{equation:3}
\end{equation}
where `$\log$' means natural logarithm. The term $M_{X,Y}$ is a measure of dependence between $X$ and $Y$.
Indeed, (\ref{equation:3}) defines a premetric, since $M_{X,Y}\geq 0$, with $M_{X,Y}=0$ if and
only if $X$ and $Y$ are independent. Roughly speaking, it measures how far $X$ and $Y$ are  from
being independent, in the sense that high values of $M_{X,Y}$ correspond to a strong dependence
between $X$ and $Y$. Moreover, $M_{X,Y}$ is in general finite and is invariant under linear
transformations. We recall that the mutual information can be expressed in terms of entropies as 
follows (see, for example, Ebrahimi {\em et al}.\ \cite{ESS2010}):
\begin{equation}
 M_{X,Y}=H_X+H_Y-H_{X,Y},
 \label{equation:4}
\end{equation}
where $H_X$ is the differential entropy of $X$, defined by
$H_X=-\int_0^{+\infty}f_X(x)\,\log {f_X(x)}\,{\rm d}x$, $H_Y$ is similarly defined, and
\begin{equation}
 H_{X,Y}=-\int_0^{+\infty}{\rm d}x\int_0^{+\infty}f(x,y)\,
 \log {f(x,y)}\,{\rm d}y
 \label{eq:HXY}
\end{equation}
is the differential entropy of $(X,Y)$. We recall that $H_X$ measures the `uniformity' of the distribution 
of $X$, i.e.\ how the distribution spreads over its domain, and is irrespective of the locations of 
concentration. High values of $H_X$ correspond to a low concentration of the probability mass of $X$.
\par
The reliability analysis of a system composed of two items involves the general setting by which they 
are inspected at possibly different times $s$ and $t$. Assuming that the random variables $X$ and $Y$ 
describe the failure times of the two items, the following conditional random vectors thus deserve interest, 
for $s,t\geq 0$,
\begin{eqnarray}
 && \hspace{-1cm} [(X,Y)\,|\,X\leq s,Y\leq t]   \quad \hbox{if both items failed before inspection},  
 \label{eq:domains1} \\
 && \hspace{-1cm} [(X,Y)\,|\,X>s,Y>t]  \quad \hbox{if no item failed before inspection},
  \label{eq:domains2} \\
 && \hspace{-1cm} [(X,Y)\,|\,X\leq s,Y>t]   \quad \hbox{if only the first item failed before inspection}, 
  \label{eq:domains3} \\
 && \hspace{-1cm} [(X,Y)\,|\,X>s,Y\leq t]  \quad  \hbox{if only the second item failed before inspection}.
 \label{eq:domains4} 
\end{eqnarray}
The probability of the conditional events considered above will be denoted as 
$$
\begin{array}{ll}
 F(s,t)=\mathbb P(X\leq s,Y\leq t), & \barF(s,t)=\mathbb P(X>s,Y>t), \\
F^{-,+}(s,t)=\mathbb P(X\leq s,Y>t), &  F^{+,-}(s,t)=\mathbb P(X>s,Y\leq t),  \\
\end{array}
$$
so that $F(s,t)+\barF(s,t)+F^{-,+}(s,t)+F^{+,-}(s,t)=1$. In order to introduce certain dynamic 
entropies, we now consider the following functions. \\
(i) The density of $[(X,Y)\,|\,X\leq s,Y\leq t]$ for all $s,t\geq 0$ such that $F(s,t)>0$, 
\begin{equation}
 \tilde f_{X,Y}(x,y;s,t)=\frac{f(x,y)}{F(s,t)},
 \qquad 0\leq x\leq s,
 \quad 0\leq y\leq t.
  \label{equation:27}
\end{equation}
(ii) The density of $[(X-s, Y-t)\,|\,X>s,Y>t]$ for all $s,t\geq 0$ such that $\barF(s,t)>0$, 
\begin{equation}
 f_{X,Y}(x,y;s,t)=\frac{f(x+s,y+t)}{\barF(s,t)},
 \qquad x\geq 0,
 \quad y\geq 0.
 \label{equation:9}
\end{equation}
(iii) The density of $[(X,Y-t)\,|\,X\leq s,Y> t]$,  for all $s,t\geq 0$ such that $F^{-,+}(s,t)>0$, 
$$
 f_{X,Y}^{-,+}(x,y;s,t)=\frac{f(x,y+t)}{F^{-,+}(s,t)}, \qquad 0\leq x\leq s, \quad y\geq 0.
$$
(iv) The density of $[(X-s,Y)\,|\,X> s,Y\leq t]$  for all $s,t\geq 0$ such that  $F^{+,-}(s,t)>0$, 
$$
 f_{X,Y}^{+,-}(x,y;s,t)=\frac{f(x+s,y)}{F^{+,-}(s,t)}, \qquad x\geq 0, \quad 0\leq y\leq t.
$$
Hence, in analogy with (\ref{eq:HXY}) we can now introduce the following entropies, for $s,t\geq 0$: 
\begin{equation}
 \tilde H_{X,Y}(s,t)=-\int_0^s \,{\rm d}x\int_0^t\tilde f_{X,Y}(x,y;s,t)\log \tilde f_{X,Y}(x,y;s,t)\,{\rm d}y,
 \label{eq:tildeHXY}
\end{equation}
\begin{equation}
H_{X,Y}(s,t)=-\int_0^{+\infty}{\rm d}x\int_0^{+\infty}f_{X,Y}(x,y;s,t)\,\log f_{X,Y}(x,y;s,t)\,{\rm d}y,
 \label{equation:15}
\end{equation}
$$
 H_{X,Y}^{-,+}(s,t)=-\int_0^s{\rm d}x\int_0^{+\infty}f_{X,Y}^{-,+}(x,y;s,t)\log f_{X,Y}^{-,+}(x,y;s,t)\,{\rm d}y,
$$
$$
 H_{X,Y}^{+,-}(s,t)=-\int_0^{+\infty}{\rm d}x\int_0^t f_{X,Y}^{+,-}(x,y;s,t)\log f_{X,Y}^{+,-}(x,y;s,t)\,{\rm d}y.
$$
\begin{remark}\label{rem:entr4}\rm
The entropy (\ref{eq:HXY}) can be expressed in terms of the entropies given above;
indeed, for all $s,t\geq 0$,
\begin{equation}
 \begin{split}
 H_{X,Y} & ={\cal H}[F(s,t),\barF(s,t),F^{-,+}(s,t),F^{+,-}(s,t)] \\
 & +F(s,t) \tilde H_{X,Y}(s,t)+\barF(s,t) H_{X,Y}(s,t) \\
 & +F^{-,+}(s,t)H_{X,Y}^{-,+}(s,t)+F^{+,-}(s,t)H_{X,Y}^{+,-}(s,t),
\end{split}
\label{eq:HHXY}
\end{equation}
where ${\cal H}[p_1,\ldots,p_n]:=-\sum_{i=1}^n p_i\log p_i$ denotes  the entropy
of a discrete probability distribution.
\end{remark}
\par
We recall that (\ref{eq:HHXY}) is the two-dimensional analogue of [8, Proposition 2.1].
It holds due to the partitioning property of the Shannon entropy (see, for example, Eq.\ (24)
of Ebrahimi {\em et al.}\ \cite{EKS2007} for another application of such a property). It expresses
that the uncertainty about the failure times of two items can be decomposed in five terms.
The first term conveys the uncertainty of whether the items failed before or after their inspection times, 
the other terms give the uncertainties about the failure times in the domains specified in
(\ref{eq:domains1})-(\ref{eq:domains4}), given that 
the items failed in the corresponding regions. Note that (\ref{eq:HHXY}) is in agreement 
with some remarks provided in [14, Section 4.4]. 
\par
We are now able to study the dynamic mutual information for the cases introduced in this section. 
\section{Mutual information for past lifetimes}
In various contexts the uncertainty is not necessarily related to the future but may refer to 
the past. For instance, if a system is observed at an inspection time $t$ and  is found 
failed, then the uncertainty relies on the past, i.e.\ on which instant in $(0, t)$ it failed. 
Several papers have been devoted to the investigation of information measures concerning 
past lifetimes. We recall, for instance, the univariate past entropy defined in  \cite{DiCrLo2002}. 
Some properties and generalizations have also been investigated   in 
\cite{NaPa2006}, \cite{NaPa2006b},  \cite{KuNaMa2010}, and  \cite{MiYa2011}.
\par
In this section we introduce the mutual information for the bivariate past lifetimes defined in 
(\ref{eq:domains1}). To this aim we consider the marginal past lifetimes
\begin{equation}
 [X\,|\,X\leq s,Y\leq t], \qquad [Y\,|\,X\leq s,Y\leq t],
 \qquad  s, t\geq 0,
 \label{eq:defxeypcond}
\end{equation}
having PDFs
\begin{equation}
 \tilde f_X(x;s,t):=\frac{1}{F(s,t)}\,\frac{\partial}{\partial x}F(x,t)
 =\frac{1}{F(s,t)}\,\int_0^t f(x,y)\,{\rm d}y,
 \qquad 0\leq x\leq s,
 \label{equation:23}
\end{equation}
\begin{equation}
 \tilde f_Y(y;s,t) :=\frac{1}{F(s,t)}\,\frac{\partial}{\partial y}F(s,y)
 =\frac{1}{F(s,t)}\,\int_0^s f(x,y)\,{\rm d}x,
 \qquad 0\leq y\leq t,
 \label{equation:24}
\end{equation}
for $s,t\geq 0$ such that $F(s,t)>0$. In analogy with (\ref{equation:3}), 
we are now able to define the following new information measure, named 
{\em bivariate dynamic past mutual information\/}:
\begin{equation}
\tilde M_{X,Y}(s,t):=\int_0^s{\rm d}x\int_0^t \tilde f_{X,Y}(x,y;s,t)\,
 \log\frac{\tilde f_{X,Y}(x,y;s,t)}{\tilde f_X(x;s,t)\,\tilde f_Y(y;s,t)}\,{\rm d}y
 \label{equation:29}
\end{equation}
for $s,t\geq 0$ such that $F(s,t)>0$, where the involved densities are given in 
(\ref{equation:27}), (\ref{equation:23}), and (\ref{equation:24}). This is a nonnegative function 
which measures the dependence between the past lifetimes of $X$ and $Y$ conditional
on $\{X\leq s,Y\leq t\}$.
\begin{remark}\rm
Similarly to (\ref{equation:4}), for $s,t\geq 0$ the following identity holds:
$$
\tilde M_{X,Y}(s,t)=\tilde H_X(s,t)+\tilde H_Y(s,t)-\tilde H_{X,Y}(s,t),
$$
where
$$
 \tilde H_X(s,t)=-\int_0^s \tilde f_X(x;s,t)\log \tilde f_X(x;s,t)\,{\rm d}x
$$
and 
$$
 \tilde H_Y(s,t)=-\int_0^t \tilde f_Y(y;s,t)\log \tilde f_Y(y;s,t)\,{\rm d}y
$$
are the entropies of the marginal past lifetimes introduced in (\ref{eq:defxeypcond}), 
and where $\tilde H_{X,Y}(s,t)$ is defined in (\ref{eq:tildeHXY}). 
\end{remark}
\par
Let us now obtain some bounds.
\begin{proposition}\label{prop:boundpM}
For $s,t\geq 0$ such that $F(s,t)>0$, let
\begin{equation}
 \tilde a(x,y;s,t)
 := \frac{f(x,y)}{\int_0^t f(x,y)\,{\rm d}y\int_0^s f(x,y)\,{\rm d}x}
 \qquad 0\leq x\leq s,
 \quad 0 \leq y\leq t.
 \label{eq:defatilde}
\end{equation}
If
$$
 \tilde a(x,y;s,t)\leq (\geq) \tilde a(s,t;s,t)
 \qquad
 \hbox{for all $\,0\leq x\leq s$ and $\,0 \leq y\leq t$,}
$$
then the following upper [lower] bound holds:
\begin{equation}
\tilde M_{X,Y}(s,t)\leq (\geq)  \log \tilde a(s,t;s,t)+\log F(s,t).
 \label{equation:33}
\end{equation}
\end{proposition}
\begin{proof}
From (\ref{equation:29}), making use of (\ref{equation:27}), (\ref{equation:23}) and (\ref{equation:24}), 
we have
\begin{equation}
 \tilde M_{X,Y}(s,t)=\frac{1}{F(s,t)}\int_0^s{\rm d}x\int_0^t f(x,y) \log \tilde a(x,y;s,t)\,{\rm d}y+\log F(s,t).
 \label{eq:tildeMst2}
\end{equation}
Hence, from (\ref{eq:defatilde}) we immediately obtain   (\ref{equation:33}).
\end{proof}
\begin{example}
{\rm
Let $(X,Y)$ be a random vector with joint PDF and distribution function
$$
 f(x,y)=x+y,
 \qquad
 F(x,y)=\frac{x y(x+y)}{2},
 \qquad 0\leq x\leq 1, \; 0\leq y\leq 1.
$$
Since, for $0<s<1$ and $0<t<1$,
$$
 \tilde a(x,y;s,t)= \frac{4(x+y)}{s t (t+2 x)(s+2 y)},
 \qquad 0\leq x\leq 1,\;0\leq y\leq 1,
$$
from (\ref{eq:tildeMst2}) we have
\begin{equation}
\begin{split}
\tilde M_{X,Y}(s,t)
& =\log \frac{st(s+t)}{2}
+\frac{1}{st(s+t)}\Big\{ st(s+t) \log\frac{4} {s t}\\
& +\frac{1}{6}[-2s^3\log s-2t^3\log t+2(s+t)^3\log(s+t)-5st(s+t)] \\
& +\frac{t}{4}[2s(s+t)+t^2\log t-(t+2s)^2\log (t+2s)] \\
& +\frac{s}{4}[2t(s+t)+s^2\log s-(s+2t)^2\log (s+2t)]\Big\}.
\end{split}
\label{eq:MtildeEsxpy}
\end{equation}
For any fixed $t\in (0,1)$, it follows that $\tilde M_{X,Y}(s,t)$ is increasing for $s\in(0,t]$
and, thus, attains the maximum for $s=t$, with
$$
 \tilde M_{X,Y}(t,t)=\frac{2+40\log 2-27\log 3}{12}=0.0053, \qquad t\in (0,1).
$$
The plot of $\tilde M_{X,Y}(s,t)$ is given in Figure \ref{fig:1}. See [14, Example 1]
for other results on the information content of the bivariate 
distribution considered in this example.
}\end{example}
%
\begin{figure}[t]
\begin{center}
\epsfxsize=8.5cm
\epsfbox{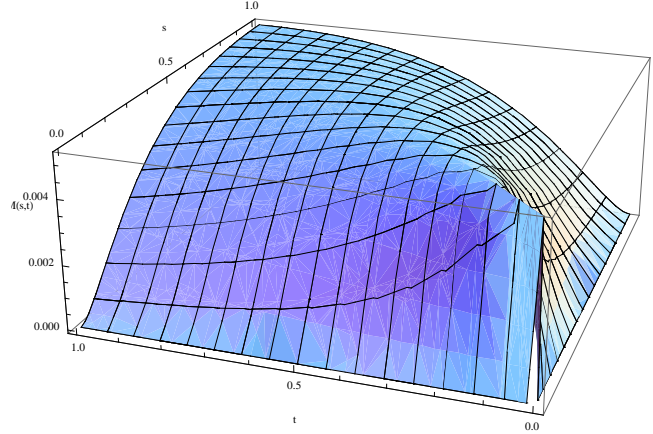}
\caption{\small{Plot of the past mutual information given in (\ref{eq:MtildeEsxpy}).}}
\label{fig:1}
\end{center}
\end{figure}
%
\par
Let us now recall that the reversed hazard rate of a random lifetime $X$ is given by
$\tau_X(x)=-({\rm d}/{\rm d}x)\log F_X(x)=f_X(x)/F_X(x)$ for all $x$ such that $0<F_X(x)<1$, where  
$F_X(x)=\mathbb P(X\leq x)$. 
\begin{remark}\label{remark:rrates}
\rm
The argument of the logarithm in (\ref{equation:29}) can be viewed as a local dynamic measure
of dependence between $X$ and $Y$. Indeed, due to (\ref{equation:27}), (\ref{equation:23}), 
and (\ref{equation:24}), we have:
$$
 \frac{\tilde f_{X,Y}(x,y;s,t)}{\tilde f_X(x;s,t)\,\tilde f_Y(y;s,t)}
 =\frac{\tau_{\tilde X_s}(x\,|\,Y=y)}{\tau_{\tilde X_s}(x\,|\,Y\leq t)},
$$
where $\tau_{\tilde X_s}(x\,|\,B)$ is the conditional reversed hazard rate of $\tilde X_s:=[X\,|\,X\leq s]$ given $B$.
\end{remark}
\section{Mutual information for residual lifetimes}
The uncertainty about the remaining lifetime in reliability systems is often measured by means of the 
differential entropy of residual lifetimes; see \cite{BNRdA2004}, \cite{Ebr1996}, and \cite{EbrPel1995}. 
Recent contributions on the entropy 
of residual lifetimes are given in \cite{AAMMB2010}. Other dynamic 
information measures involving conditional lifetimes have been proposed and studied
in \cite{AsZo2007}, \cite{DiCrLo2009}, and \cite{NaDAA2010}.
For a random vector $(X,Y)$ with nonnegative absolutely continuous components, 
Di Crescenzo {\em et al.}\ \cite{DiCrLoNa2004} studied the mutual information 
of the residual lifetimes $[X-t\,|\,X>t]$ and $[Y-t\,|\,Y>t]$ at the same age. 
\par
In this section, with reference to (\ref{eq:domains2}), we investigate the mutual information 
of the residual lifetimes at different ages, i.e. 
\begin{equation}
 [X-s\,|\,X>s,Y>t], 
 \qquad
 [Y-t\,|\,X>s,Y>t], 
 \qquad  s, t\geq 0.
 \label{equation:6}
\end{equation}
For all $s,t\geq 0$ such that $\barF(s,t)>0$, 
the random variables (\ref{equation:6}) possess densities
\begin{equation}
 f_X(x;s,t)=\frac{1}{\barF(s,t)} \left[-\frac{\partial}{\partial u}\barF(u,t)\right]_{u=x+s}
 =\frac{1}{\barF(s,t)}\,\int_t^{+\infty}f(x+s,y)\,{\rm d}y,
 \qquad x\geq0
 \label{equation:7}
\end{equation}
and
\begin{equation}
 f_Y(y;s,t)=\frac{1}{\barF(s,t)} \left[-\frac{\partial}{\partial v}\barF(s,v)\right]_{v=y+t}
 =\frac{1}{\barF(s,t)}\,\int_s^{+\infty}f(x,y+t)\,{\rm d}x,
 \qquad y\geq0.
 \label{equation:8}
\end{equation}
According to  (\ref{equation:3}) we thus introduce the 
{\em bivariate dynamic residual mutual information\/}, for $s,t\geq 0$ such that $\barF(s,t)>0$, 
\begin{equation}
 M_{X,Y}(s,t):=\int_0^{+\infty}{\rm d}x\int_0^{+\infty}f_{X,Y}(x,y;s,t)\,
 \log\frac{f_{X,Y}(x,y;s,t)}{f_X(x;s,t)\,f_Y(y;s,t)}\,{\rm d}y,
 \label{equation:10}
\end{equation}
the involved densities being defined in  (\ref{equation:9}), (\ref{equation:7}) and (\ref{equation:8}).  
Since $X$ and $Y$ describe the random lifetimes of two systems, $M_{X,Y}(s,t)$ measures the 
dependence between their remaining lifetimes at different ages $s$ and $t$. See the analogy between 
(\ref{equation:10}) and the mutual information of $[(X, Y)\,|\,X>s,Y>t]$ given in 
[14, Equation (1)]. We remark that other types of dynamic information 
measures for bivariate distributions have been studied by Sunoj and Linu \cite{SuLi2012}.
\par
Moreover, in agreement with (\ref{equation:4}), the mutual information $M_{X,Y}(s,t)$ satisfies
the following identity (see [14, Equation (13)]):
\begin{equation}
 M_{X,Y}(s,t)=H_X(s,t)+H_Y(s,t)-H_{X,Y}(s,t),
 \qquad s,t\geq 0,
 \label{equation:12}
\end{equation}
where $H_{X,Y}(s,t)$ is defined in (\ref{equation:15}) and 
\begin{equation}
\begin{split}
 & H_X(s,t)=-\int_0^{+\infty}f_X(x;s,t)\,\log f_X(x;s,t)\,{\rm d}x, \\
 & H_Y(s,t)=-\int_0^{+\infty}f_Y(y;s,t)\,\log f_Y(y;s,t)\,{\rm d}y
\end{split} 
 \label{equation:13}
\end{equation}
denote the entropies of the residual lifetimes (\ref{equation:6}), for $s,t\geq 0$. 
Various other results have been pinpointed in  \cite{EKS2007},
such as the following property: if $X$ and $Y$ are exchangeable then
$M_{X,Y}(s,t)=M_{X,Y}(t,s)$ for all $s,t\geq 0$.
\par
We recall that the hazard rate of a random lifetime $X$  is given by
$h_X(x)=-({\rm d}/{\rm d}x)\log [1-F_X(x)]=f_X(x)/[1-F_X(x)]$ for all $x$ such that $0<F_X(x)<1$.
\begin{remark}\label{remark:rates}
\rm
Similarly as in Remark \ref{remark:rrates},
the argument of the logarithm in  (\ref{equation:10}) can be viewed as a local dynamic measure
of dependence between $X$ and $Y$. Indeed, from (\ref{equation:9}), (\ref{equation:7}), and  
(\ref{equation:8}), we have
\begin{equation}
 \frac{f_{X,Y}(x,y;s,t)}{f_X(x;s,t)\,f_Y(y;s,t)}
 =\frac{h_{X_s}(x\,|\,Y=y+t)}{h_{X_s}(x\,|\,Y>t)},
 \label{eq:misurahXs}
\end{equation}
where $h_{X_s}(x\,|\,B)$ is the conditional hazard rate of $X_s:=[X-s\,|\,X>s]$ given $B$.
Moreover, the right-hand-side of  (\ref{eq:misurahXs}) is a suitable extension of association 
measures that are often employed in reliability theory (see, for example \cite{Gu2008} 
and references therein).
\end{remark}
\par
The following result is analogous to Proposition \ref{prop:boundpM}.
\begin{proposition}\label{prop:boundM}
For $s,t\geq 0$ such that $\barF(s,t)>0$, let 
\begin{equation}
 a(x,y;s,t):=\frac{f(x,y)}{\int_t^{+\infty}f(x,v)\,{\rm d}v \,\int_s^{+\infty}f(u,y)\,{\rm d}u}
 \qquad x\geq s,\;\; y\geq t.
  \label{equation:defa}
\end{equation}
If
\begin{equation}
  a(x,y;s,t)\leq [\geq]\;  a(s,t;s,t) \qquad
 \hbox{for all $\,x\geq s$ and $\,y\geq t$,}
  \label{equation:16}
\end{equation}
then the following upper [lower] bound holds:
\begin{equation}
 M_{X,Y}(s,t)\leq [\geq]\; \log a(s,t;s,t)+\log \barF(s,t).
 \label{eq:boundM}
\end{equation}
\end{proposition}
\begin{proof}
Due to (\ref{equation:9}), (\ref{equation:7}), and (\ref{equation:8}),
from (\ref{equation:10}) we obtain the following alternative expression for $M_{X,Y}(s,t)$, $s,t\geq 0$:
\begin{equation}
 M_{X,Y}(s,t)=\frac{1}{\barF(s,t)}\int_s^{+\infty}{\rm d}x \int_t^{+\infty}f(x,y)\,
 \log a(x,y;s,t)\,{\rm d}y +\log \barF(s,t).
 \label{equation:11}
\end{equation}
The proof then immediately follows by use of (\ref{equation:16}) in the
right-hand side of (\ref{equation:11}).
\end{proof}
\begin{example}\rm
Let $(X,Y)$ be a random vector with joint PDF
$$
 f(x,y)=\frac{\theta}{\Gamma(0, {1}/{\theta })}\,\exp\left\{-\frac{1}{\theta }(1+\theta x)(1+\theta y)\right\},
 \qquad x,y\geq 0,
$$
with $\theta>0$, and where $\Gamma(a,z)=\int_z^{+\infty}t^{a-1}\,e^{-t}\,{\rm d}t$ 
is the incomplete Gamma function. Since
$$
 \barF(x,y)=\frac{\Gamma\left(0, ({1}/{\theta })(1+\theta x)(1+\theta y)\right)}{\Gamma(0, {1}/{\theta})},
 \qquad x,y\geq 0,
$$
from (\ref{equation:9}), we have, for $s,t\geq 0$,
$$
 f_{X,Y}(x,y;s,t) =\frac{\theta\, \exp\{- ({1}/{\theta})[1+\theta (x+s)]\,[1+\theta(y+t)]\}}
 {\Gamma\left(0, ({1}/{\theta })(1+\theta s)(1+\theta t)\right)},
 \qquad x,y\geq 0.
$$
Hence, recalling (\ref{equation:defa}), after some calculations we obtain, for $x,y\geq 0$,
$$
 a(x,y;s,t) =\frac{1} {\theta}\,\Gamma\Big(0,\frac{1}{\theta}\Big)\,
  (1+\theta\,x)(1+\theta\,y) \exp\left\{\frac{1}{\theta}\left[1+\theta(s+t)+\theta^2(t x+s y -xy)\right]\right\}.
$$
This expression allows to  evaluate $M_{X,Y}(s,t)$ numerically, by use of (\ref{equation:11}).
Other properties of dynamic measures concerning this case
are given in [14, Section 4.4].
\end{example}
\par
In the following proposition we show a relation between the bivariate dynamic residual and
past mutual information of symmetric random vectors.
\begin{proposition}\label{prop:symmMM}
If the random vector $(U,V)$ has bivariate density $f_{U,V}(x,y)$ such that, for a fixed
$(x_0,y_0)\in \mathbb{R}_+^2$,
\begin{equation}
 f_{U,V}(x,y)=f(2x_0-x, 2y_0-y)
 \qquad \hbox{for all $(x,y)\in \mathbb{R}_+^2$,}
 \label{eq:relsimm}
\end{equation}
then
$\tilde M_{U,V}(s,t)= M_{X,Y}(2x_0-s,2y_0-t)$  for all $s,t\geq 0$.
\end{proposition}
\begin{proof}
The proof follows from the definitions of $\tilde M_{U,V}$ and  $M_{X,Y}$,
since the distribution function of $(U,V)$ satisfies
$F_{U,V}(x,y)=\barF(2x_0-x, 2y_0-y)$ for all $(x,y)\in \mathbb{R}_+^2$.
\end{proof}
\begin{example}\label{esempio:alfabeta}
\rm
Let $(X,Y)$ be a random vector uniformly distributed over the domain
${\cal D}:=\{(x,y):x\geq 0, \; y\geq 0,\; \alpha x+\beta y\leq 1\}$ with $\alpha, \beta>0$.
Hence, the joint PDF and the joint survival function are given by 
$$
 f(x,y)=2 \alpha\beta,
 \qquad
 \barF(x,y)=(1-\alpha x-\beta y)^2,
 \qquad (x,y)\in {\cal D},
$$
so that, from (\ref{equation:7}), we have the density
$$
 f_X(x;s,t)=\frac{2\alpha[1-\alpha (s+x)-\beta t]}
 {(1-\alpha s-\beta t)^2},
 \qquad 0\leq x\leq \frac{1}{\alpha}-\frac{\beta}{\alpha}t-s, 
 \qquad (s,t)\in{\cal D}.
$$
Due to (\ref{equation:13}), for $(s,t)\in{\cal D}$ the entropies of the residual lifetimes are  
\begin{equation}
 H_X(s,t)= \frac{1}{2}+\log \frac{1-\alpha s-\beta t}{2\alpha},
 \qquad
 H_Y(s,t)= \frac{1}{2}+\log \frac{1-\alpha s-\beta t}{2\beta}.
 \label{equation:20}
\end{equation}
From (\ref{equation:9}), we obtain, for $(s,t)\in{\cal D}$,
$$
 f_{X,Y}(x,y;s,t)= \frac{2\alpha\beta}{(1-\alpha s-\beta t)^2},
 \qquad (x+s,y+t)\in {\cal D}.
$$
Hence, making use of (\ref{equation:15}) we obtain the entropy of 
$[(X-s, Y-t)\,|\, X > s, Y > t]$:
\begin{equation}
H_{X,Y}(s,t)= 2\log (1-\alpha s-\beta t)-\log (2\alpha\beta),
 \qquad (s,t)\in{\cal D}.
  \label{equation:22}
\end{equation}
In conclusion, recalling (\ref{equation:12}), (\ref{equation:20}) and (\ref{equation:22})
we establish that the dynamic residual mutual information of $(X,Y)$ is constant:
\begin{equation}
M_{X,Y}(s,t)=1-\log 2=0.3069, \qquad (s,t)\in{\cal D}.
\label{eq:MXYesab}
\end{equation}
Note that in this case for $(s,t)\in{\cal D}$, we have
$$
 a(x,y;s,t)=\frac{1}{2(1-\alpha x -\beta t)(1-\alpha s -\beta y)}
 \geq a(s,t;s,t)= \frac{1} {2(1-\alpha s-\beta t)^2};
$$
however, now the bound given in (\ref{eq:boundM}) is not useful since the right-hand side of 
(\ref{eq:boundM}) is negative. Let $(U,V)$ have density
$$
 f_{U,V}(x,y)=2\alpha\beta 
 \qquad \hbox{for $(x,y)\in \tilde D:=
 \left\{x\leq \displaystyle\frac{1}{\alpha}, \; y\leq \frac{1}{\beta}, \;\alpha x+\beta y\geq 1\right\}$,}
$$
and distribution function $F_{U,V}(x,y)=(\alpha x+\beta y-1)^2$ for $(x,y)\in \tilde D$.
Then, $(U,V)$ is symmetric to  $(X,Y)$,
in the sense that   (\ref{eq:relsimm}) holds for $(x_0,y_0)=({1}/{2\alpha}, {1}/{2\beta}).$
Hence, making use of Proposition \ref{prop:symmMM} and recalling (\ref{eq:MXYesab}),
we have $\tilde M_{U,V}(s,t)=1-\log 2=0.3069$ for $(s,t)\in \tilde D$.
\end{example}
\par
It worthwhile to remark that  the residual mutual information is constant also in other cases.
See [14, Section 3.1] for various comments on the memoryless property
and related information notions. We recall that if the survival function of a non-negative continuous
vector variable $(X,Y)$ satisfies $\barF(x+t,y+t)=\barF(x,y)\,\barF(t,t)$, $\forall \;x,y,t\geq 0$,
then $(X,Y)$ is said to possess the bivariate lack of memory (BLM) property; see, for instance,
\cite{Ro2002}. It thus follows that if $(X,Y)$ has the BLM property, then $M_{X,Y}(t,t)$ does
not depend on $t$. For instance, the bivariate Block-Basu density and the bivariate
Freund density have the BLM  property. See also \cite{Pe2008} for the weak
multivariate lack of memory property within a stochastic model that will be discussed hereafter.
%
\subsection{Dynamic mutual information for time-transformed exponential model}
We recall that a pair of random lifetimes $(X,Y)$ is said to follow the
time-transformed exponential (TTE) model if its joint survival function may be expressed in the following way:
\begin{equation}
 \barF(s,t)=\barW[R_1(s)+R_2(t)],
 \qquad \hbox{for all }s,t\geq 0,
 \label{eq:FstTTE}
\end{equation}
where $\barW:[0,+\infty)\to [0,1]$ is a continuous, convex, and strictly decreasing
survival function, such that $\barW(0)=1$ and $\displaystyle\lim_{r\to +\infty}\barW(r)=0$,
and where  $R_i:[0,+\infty)\to [0,+\infty)$ is a continuous and strictly increasing function,
such that $R_i(0)=0$ and $\displaystyle\lim_{t\to+\infty}R_i(t) = +\infty$ for $i=1,2$.
Clearly, functions $\barW$ and $R_i$, $i=1,2$, provide the time transform and the accumulated
hazards, respectively. Note that the marginal survival functions are given by
$\barF_X(s)= \barW[R_1(s)]$, $s\geq 0$,
and $\barF_Y(t)= \barW[R_2(t)]$, $t\geq 0$. Moreover, if $R_1$ and $R_2$ are identical functions,
then $X$ and $Y$ are exchangeable. The TTE model allows us to study the essential ageing properties
of lifetimes $(X,Y)$ by separating ageing property and dependence and, thus, it deserves wide interest
in reliability theory and survival analysis. Various properties and applications of such semiparametric
model have been investigated recently in, for example,  \cite{BaSp2005}, \cite{LiLin2011}, 
\cite{MuPeRo2010}, \cite{Pe2008}, and \cite{YLB2014}.
\par
Hereafter, we investigate the bivariate dynamic residual mutual information within the TTE model.
\begin{proposition}\label{prop:TTE}
If the survival function of $(X,Y)$ satisfies the TTE model as specified in (\ref{eq:FstTTE}), then
for all $s,t\geq 0$, 
\begin{equation}
\begin{split}
M_{X,Y}(s,t)&= \frac{1}{\barW[R_1(s)+R_2(t)]}  \\
& \times \int_{R_1(s)}^{+\infty}{\rm d}u
\int_{R_2(t)}^{+\infty}\barW''[u+v] \,
 \log \frac{\barW''[u+v]\barW[R_1(s)+R_2(t)]}{\barW'[u+R_2(t)] \,\barW'[R_1(s)+v]}\,{\rm d}v.
\end{split}
\label{eq:MTTE}
\end{equation}
\end{proposition}
\begin{proof}
Let $s,t\geq 0$. From (\ref{eq:FstTTE}), it follows that 
$$
 f(s,t)=\barW''[R_1(s)+R_2(t)]R_1'(s)R_2'(t).
$$
Hence, from (\ref{equation:9}), (\ref{equation:7}), and (\ref{equation:8}), we have
\begin{equation}
 f_X(x;s,t)=-\frac{\barW'[R_1(x+s)+R_2(t)]R_1'(x+s)}{\barW[R_1(s)+R_2(t)]},
 \qquad x\geq 0,
  \label{equation:18}
\end{equation}
\begin{equation}
 f_Y(y;s,t)=-\frac{\barW'[R_1(s)+R_2(y+t)]R_2'(y+t)}{\barW[R_1(s)+R_2(t)]},
 \qquad y\geq 0,
  \label{equation:19}
\end{equation}
and
\begin{equation}
f_{X,Y}(x,y;s,t)=\frac{\barW''[R_1(x+s)+R_2(y+t)]R_1'(x+s)R_2'(y+t)}{\barW[R_1(s)+R_2(t)]},
\qquad x,y\geq 0.
  \label{equation:17}
\end{equation}
Finally,  (\ref{eq:MTTE}) follows by substituting the above densities in the right-hand side 
of (\ref{equation:10}), and by setting $u=R_1(x+s)$ and $v=R_2(y+t)$.
\end{proof}
\par
The following result can be obtained by means of straightforward calculations.
\begin{corollary}\label{cor:TTE}
Let $(X,Y)$ satisfy the assumptions of Proposition \ref{prop:TTE}. If
$$
 \barW(x)=(1+x)^{-r}, \quad x\geq 0,
 \qquad
 R_1(s)=\alpha s, \quad s\geq 0,
 \qquad
 R_2(t)=\beta t, \quad t\geq 0,
$$
with $r,\alpha, \beta>0$, then
$$
 M_{X,Y}(s,t)=-\frac{1}{r+1}+\log\frac{r+1}{r},
 \qquad s,t\geq 0.
$$
\end{corollary}
\par
From Corollary \ref{cor:TTE} we show that if $(X,Y)$ has bivariate Lomax (Pareto type II) joint survival
function then $M_{X,Y}(s,t)$ is constant (see also [14, Section 5.2]).
Note that in this case $a(x,y;s,t)$ is not monotone; so that the bound (\ref{equation:16}) is not useful.
\subsection{Dynamic mutual information for truncated TTE model}
We now consider a TTE model for {\em truncated\/} random lifetimes $(X,Y)$. Specifically, 
we assume such that the nonnegative random variables $X$ and $Y$  are upper bounded 
through a suitable function. Unlike the previous section, we now assume that 
$\barW(r)$ is a continuous, convex and strictly decreasing one-dimensional survival function 
for all $r\in [0, \omega]$, where $\omega$ is a fixed positive real number, such that 
$\barW(0)=1$ and $\barW(\omega)=0$. Moreover, $R_1(\cdot)$ and $R_2(\cdot)$ are 
continuous and strictly increasing functions such that $R_1(0)=R_2(0)=0$, and the set
$$
 D_{\omega}:= \{(s,t)\in\mathbb{R}^2: s\geq 0, \; t\geq 0, \; R_1(s)+R_2(t)\leq \omega \}
$$
is not empty. Hence, there exists a continuous and strictly decreasing function $t= \ell_{\omega}(s)$,
defined for $0\leq s\leq R_1^{-1}(\omega)$, and such that $R_1(s)+R_2(\ell_{\omega}(s))=\omega$
for all $s\in [0, R_1^{-1}(\omega)]$, with $\ell_{\omega}(0)=R_2^{-1}(\omega)$ and
$\ell_{\omega}(R_1^{-1}(\omega))=0$. These assumptions thus lead to the following
{\em truncated TTE model\/} for the joint survival function of $(X,Y)$:
\begin{equation}
 \barF(s,t)=\barW[R_1(s)+R_2(t)]
 \qquad \hbox{for all }(s,t)\in D_{\omega}.
 \label{eq:FstTTEtrunc}
\end{equation}
Similarly to Proposition \ref{prop:TTE}, we thus have the following result for the
dynamic residual mutual information within the above model.
\begin{proposition}\label{prop:TTEtr}
If the joint survival function of $(X,Y)$ satisfies the TTE model as specified in 
(\ref{eq:FstTTEtrunc}), with $\barW'(\omega)=0$, then, for all $s,t\in D_{\omega}$, 
\begin{equation*}
\begin{split}
M_{X,Y}(s,t)&= \frac{1}{\barW[R_1(s)+R_2(t)]}  \\
& \times \int_{R_1(s)}^{\omega-R_2(t)} {\rm d}u \int_{R_2(t)}^{\omega-u} \barW''[u+v] \,
 \log \frac{\barW''[u+v]\barW[R_1(s)+R_2(t)]}{\barW'[u+R_2(t)] \,\barW'[R_1(s)+v]}\,{\rm d}v.
\end{split}
\end{equation*}
\end{proposition}
\begin{proof}
Under the given assumptions the densities  in  (\ref{equation:9}), (\ref{equation:7}), 
and (\ref{equation:8}) can still be expressed respectively as in (\ref{equation:18}) for
$0\leq x\leq R_1^{-1}(\omega-R_2(t))-s$, as in (\ref{equation:19}) for
$0\leq y\leq R_2^{-1}(\omega-R_1(s))-t$, and as in (\ref{equation:17}) for all nonnegative
$x,y$ such that $R_1(x+s)+R_2(y+t)\leq \omega$. Note that the assumption $\barW'(\omega)=0$
is essential  to ascertain that the integral of $f_{X,Y}(x,y;s,t)$ is unity.
The proof thus proceeds similarly as that of Proposition \ref{prop:TTE}.
\end{proof}
\par
The following result can be obtained via direct calculations.
\begin{corollary}\label{cor:TTEtrunc}
Let $(X,Y)$ satisfy the assumptions of Proposition \ref{prop:TTEtr}. If
$$
 \barW(x)=\left(\frac{x}{\omega}-1\right)^{2}, \quad 0\leq x\leq \omega,
$$
then
$M_{X,Y}(s,t)=1-\log{2}=0.3069$, $(s,t)\in  D_{\omega}$. 
\end{corollary}
%
\section{Dynamic mutual information for ordered data}
The approach developed in the previous sections can also be adopted  to study the mutual 
information in the presence of conditioning expressed as in (\ref{eq:domains3}) and 
(\ref{eq:domains4}). Here we restrict ourselves to consider models based on ordered data, 
with an application to order statistics. For $n\geq 2$, consider a system with  $n$  
components, having  independent and identically distributed random lifetimes. 
Assume that the failures of the components are observed upon a test. Suppose that the 
$i$th failure occurs before time $s$  and  $n-j+1$ $(j>i)$ components are still alive at 
time $t$, with $0<s<t$. For $1\leq i<j\leq n$ we can define the following random variables:
\begin{equation}
 T_{i,j:n}(s, t)=[(X_{i:n}, X_{j:n})\,|\,X_{i:n}\leq s, X_{j:n}> t], \qquad 0< s<t,
 \label{e1}
\end{equation}
where $X_{r:n}$ denotes the $r$th order statistic. We recall that Ebrahimi {\em et al.}\ \cite{ESZ2004}
defined and studied mutual information between consecutive ordinary order statistics.
\par
Let us now define dynamic mutual information measures for order  statistics. As a case study, we
consider \eqref{e1} for  $i=1$ and $j=n$, i.e.\ we assume that the first failure occurs before time $s$,
and  the last failure occurs after time $t$. Then the joint PDF of $T_{1,n:n}(s, t)$
and the marginal PDFs of $[X_{1:n}\,|\,X_{1:n}\leq s, X_{n:n}> t]$ and
$[X_{n:n}\,|\,X_{1:n}\leq s, X_{n:n}> t]$ are needed. Let $f(x)$ and $F(x)$ denote respectively the
common PDF and the distribution function of the components' lifetimes. Since
(see, for example, \cite{DaNa2003})
$$
  f_{1,n:n}(x,y)=n(n-1)[F(y)-F(x)]^{n-2}f(x)f(y), \qquad 0<x<y<+\infty,
$$
for $0<s<t$, we have
\begin{eqnarray}
 \mathbb P(X_{1:n}\leq s, X_{n:n}> t) \!\!\! &=& \!\!\!
 \int_{0}^{s}{\rm d}x\int_{t}^{+\infty} f_{1,n:n}(x,y){\rm d}y
 \nonumber \\
 &=& \!\!\!  \int_{0}^{s} nf(x)\left\{[1-F(x)]^{n-1}-[F(t)-F(x)]^{n-1} \right\}dx
 \nonumber \\
 &=& \!\!\!  1-[F(t)]^n+[F(t)-F(s)]^n-[1-F(s)]^n.
  \label{e2}
\end{eqnarray}
Let 
\begin{equation*}
  f^*_{1:n}(x; s,t)=\frac{({\partial}/{\partial x}) \mathbb P(X_{1:n}\leq x < t < X_{n:n})}
  {\mathbb P(X_{1:n}\leq s, X_{n:n}> t)}
\end{equation*}
denote the PDF of $[X_{1:n}\,|\,X_{1:n}\leq s, X_{n:n}> t]$. 
Hence, using \eqref{e2}, we obtain
\begin{equation}
  f^*_{1:n}(x; s,t)=\frac{n\left\{[1-F(x)]^{n-1}-[F(t)-F(x)]^{n-1} \right\}f(x)}
  {1-[F(t)]^n+[F(t)-F(s)]^n-[1-F(s)]^n}, \qquad  0<x<s<t.
\label{e3}
\end{equation}
Similarly, denoting the PDF of  $[X_{n:n}\,|\,X_{1:n}\leq s, X_{n:n}> t]$ by $f^*_{n:n}(y; s,t)$, we have
\begin{equation}
  f^*_{n:n}(y; s,t)=\frac{n\left\{[F(y)]^{n-1}-[F(y)-F(s)]^{n-1} \right\}f(y)}
  {1-[F(t)]^n+[F(t)-F(s)]^n-[1-F(s)]^n}, \qquad 0<s<t<y.
 \label{e4}
\end{equation}
Also, let $f^*_{1, n:n}(x, y; s,t)$ be the PDF of $T_{1,n:n}(s, t)$. Then,  it is given by
\begin{equation}
 f^*_{1, n:n}(x, y; s,t)=\frac{n(n-1)[F(y)-F(x)]^{n-2}f(x)f(y)}
 {1-[F(t)]^n+[F(t)-F(s)]^n-[1-F(s)]^n}, \qquad 0< x<s< t<y.
 \label{e5}
\end{equation}
By virtue of \eqref{e3}, \eqref{e4}, and \eqref{e5}, the dynamic mutual information of $T_{1,n:n}(s, t)$ can
thus be defined as
\begin{equation}
 M^*_{1, n:n}(s,t)=\int_{0}^{s}{\rm d}x\int_{t}^{+\infty}  f^*_{1, n:n}(x, y; s,t)
 \log \frac{f^*_{1, n:n}(x, y; s,t) }{f^*_{1:n}(x; s,t) f^*_{n:n}(y; s,t) } \,{\rm d}y, \qquad 0< s<t.
 \label{e6}
\end{equation}
Obviously, \eqref{e6} depends on $s$, $t$, $n$,  $F(s)$ and $F(t)$. Also, $M^*_{1, 2:2}(s,t)=0$ for all $0<s<t$.  However, in agreement with [17, Theorem 3.3(a)], in the following we show that $M^*_{1, n:n}(s,t)$
is distribution-free under suitable assumptions.
\par
According to the previous comments, $s$ and $t$ can be seen as inspections times for the $n$-component 
system. The knowledge of $[X_{1:n}\leq s, X_{n:n}> t]$ thus means that, upon inspection, at least one 
failed component has been detected at time $s$, and at least one  component is functioning at time $t$.
We can fix $s$ and $t$ as quantiles of $F$, say as the $p$th and $q$th quantiles, respectively, i.e.
\begin{equation}
 s=\xi_p=F^{-1}(p),  \qquad  t=\xi_q=F^{-1}(q),
 \qquad  0<p<q<1,
 \label{eqst}
\end{equation}
where $F^{-1}$   is the generalized inverse of $F$.
Denote by $H_n(p,q)$ the joint probability \eqref{e2} when $s$ and $t$ are chosen as in \eqref{eqst}, i.e.
\begin{equation}
 H_n(p,q)=1-q^n+(q-p)^n-(1-p)^n,
 \qquad  0<p<q<1.
 \label{eqdefH}
\end{equation}
Moreover, in order to show that $M^*_{1, n:n}(s,t)$ is distribution-free, for $p,q\in (0,1)$, we set
\begin{equation}
 K_n(p,q):=\int_0^p \left[(1-u)^{n-1}-(q-u)^{n-1}\right]\log\left((1-u)^{n-1}-(q-u)^{n-1}\right){\rm d}u.
 \label{eqdefK}
\end{equation}
\begin{proposition}\label{prop:M1n}
Let $n\geq 2$. If $s$ and $t$ are chosen as in \eqref{eqst}, with $0<p<q<1$, then the dynamic mutual
information of $T_{1,n:n}(s, t)$ is given by
\begin{eqnarray}
 M^*_{1, n:n}(\xi_p,\xi_q) \!\!\! &=& \!\!\! \log\Big[\frac{n-1}{n}H_n(p,q)\Big]- \frac{(n-2) (2 n-1)}{n(n-1)} 
 \nonumber \\
 \!\!\! & -& \!\!\! \frac{n}{H_n(p,q)}  \Bigg\{ K_n(p,q)+K_n(1-q,1-p)+\frac{n-2}{n}
 \nonumber \\
 \!\!\! & \times& \!\!\!     [(1-p)^n \log(1-p)+q^n \log(q)-(q-p)^n \log(q-p)] \Bigg\},
 \label{eqMn}
\end{eqnarray}
where $H_n$ and $K_n$ are given in \eqref{eqdefH} and \eqref{eqdefK}, respectively.
\end{proposition}
\begin{proof}
Equation \eqref{eqMn} follows from  \eqref{e6}, and by using densities \eqref{e3}, \eqref{e4} and \eqref{e5}.
\end{proof}
%
\begin{figure}[t]
\begin{center}
\epsfxsize=8.5cm
 \epsfbox{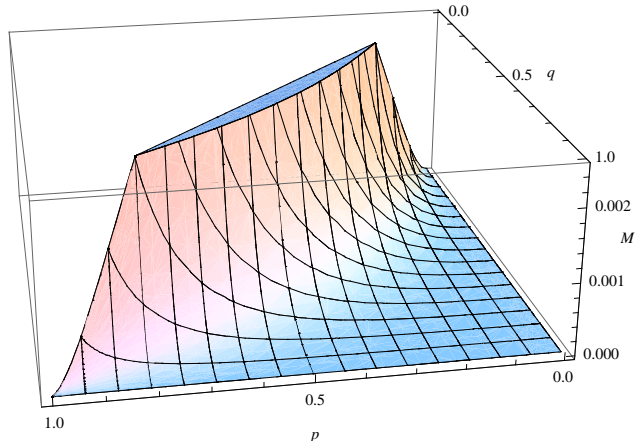}
\caption{\small{Plot of the mutual information given in \eqref{eqMn}, for $n=3$}.}
\label{fig:3}
\end{center}
\end{figure}
\begin{figure}[t]
\begin{center}
\epsfxsize=7cm
 \epsfbox{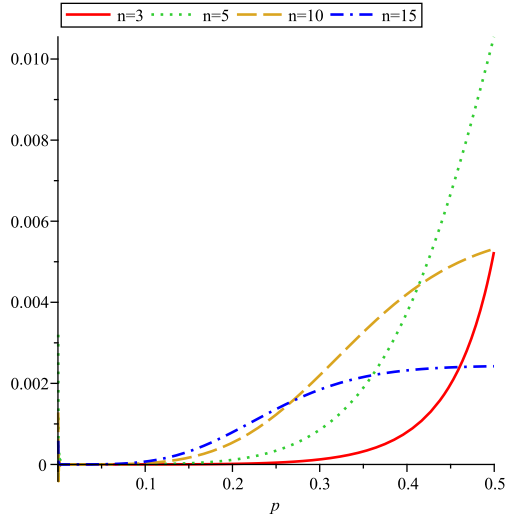}
\caption{\small{Plot of the mutual information given in \eqref{eqMn} for $n=3,5,10, 15$ and $q=1-p$}.}
\label{fig:p=1-q}
\end{center}
\end{figure}
%
\par
If $n=2$ the analysis of $T_{1,2:2}(s, t)$ is trivial. Indeed, from Proposition \ref{prop:M1n} it is not 
hard to see that $M^*_{1,2:2}(\xi_p,\xi_q)=0 $ for all $ 0<p<q<1.$  Also, a closed-form expression 
for $M^*_{1, n:n}(\xi_p,\xi_q)$ can be obtained from \eqref{eqMn} when $n=3$; however, we 
omit it being lengthy and tedious. We limit ourselves to show in Figure \ref{fig:3} the plot of 
$M^*_{1, 3:3}(\xi_p,\xi_q)$ for $0<p<q<1$. Furthermore, in Figure \ref{fig:p=1-q} we show 
the plot of $M^*_{1, n:n}(\xi_p,\xi_{1-p})$ for some selected values of $n$ and $0<p<\frac{1}{2}$, 
in the special case $q=1-p$. From Figure \ref{fig:p=1-q}, we confirm that the mutual information  
$M^*_{1, n:n}(\xi_p,\xi_{1-p})$ is increasing in $p$, as expected.
%
\section{A copula-based approach}
The copula function is an useful tool in studying the dependency in multivariate distributions; 
see, for example, \cite{DuSe2015}. Sklar's Theorem asserts that, 
given a copula $C:[0,1]^2\to [0,1]$, the joint cumulative distribution function of $(X,Y)$ 
can be written in terms of the marginals as
\begin{equation}
 F(x, y) =C(F_X(x), F_Y(y)), \qquad x,y\in\mathbb{R},
 \label{eq:relFxyC}
\end{equation}
the copula being unique if the marginals are continuous. The corresponding copula density is given by
$$
 c(u,v)=\frac{\partial^2}{\partial u\partial v }C(u, v)
 =\frac{\partial^2}{\partial u\partial v}F(F^{-1}_X(u), F^{-1}_Y(v)), \qquad u,v\in(0,1),
$$
where $F^{-1}_X$ and $F^{-1}_Y$ denote the generalized inverse of the marginals.
Thus, the joint PDF of $(X,Y)$ can be expressed as
\begin{equation}
 f(x,y)=f_X(x)f_Y(y)\,c(F_X(x), F_Y(y)) , \qquad x,y\in\mathbb{R},
 \label{eq:fxycop}
\end{equation}
so that the mutual information can be written in terms of the copula density as (see, for
example, \cite{DaDo2003})
$$
 M_{X,Y}=\int_{0}^{1} {\rm d}u\int_{0}^{1} c(u,v) \log c(u, v)\,{\rm d}v.
$$
This confirms that the mutual information does not depend on the marginal distributions, and
also that the copula entails all essential information on the dependence between $X$ and $Y$.
\par
Let us now represent the dynamic past mutual information in terms of the copula function.
We first make use of \eqref{eq:fxycop} in \eqref{equation:23} and perform the substitution 
$v=F_Y(y)$ in the integral. Moreover, similarly to \eqref{eqst}, we set
\begin{equation}
 s=\xi_p=F_X^{-1}(p), \qquad  t=\xi_q=F_Y^{-1}(q), \qquad p,q\in (0,1),
 \label{eq:stpq}
\end{equation}
so that the density of the marginal past lifetime $[X\,|\,X\leq F_X^{-1}(p), Y\leq F_Y^{-1}(q)]$ 
can be expressed as
\begin{equation}
\tilde f_X(x;\xi_p,\xi_q)=\frac{ f_X(x)}{C(p,q)}\,\int_0^q c(F_X(x), v) \,{\rm d}v ,
 \qquad 0\leq x\leq F_X^{-1}(p),
 \label{eq:ftildeX}
\end{equation}
the right-hand side of \eqref{eq:ftildeX} being a weighted density of $X$. Similarly, 
from \eqref{equation:24}, it follows that  the density of 
$[Y\,|\,X\leq F_X^{-1}(p), Y\leq F_Y^{-1}(q)]$ is given by
\begin{equation}
 \tilde f_Y(y;\xi_p,\xi_q) =\frac{f_Y(y)}{C(p, q)}\,\int_0^p c(u, F_Y(y)) \,{\rm d}u,
 \qquad 0\leq y\leq F_Y^{-1}(q).
 \label{eq:ftildeY}
\end{equation}
Finally, for the bivariate past lifetimes
\begin{equation}
 [(X,Y)\,|\,X\leq F_X^{-1}(p),Y\leq F_Y^{-1}(q)],
 \qquad p,q\in(0,1)
 \label{eq:defXYcond}
\end{equation}
the joint PDF \eqref{equation:27} becomes
\begin{equation}
 \tilde f_{X,Y}(x,y;\xi_p,\xi_q)=f_X(x)f_Y(y)\,\frac{c(F_X(x), F_Y(y))}{C(p, q)},
  \label{eq:jointftildeXY}
\end{equation}
for $0\leq x\leq F_X^{-1}(p)$ and $0\leq y\leq F_Y^{-1}(q)$.
\begin{proposition}\label{prop:copulapq}
For all $p,q\in(0,1)$ the mutual information of the bivariate past lifetimes \eqref{eq:defXYcond} is given by
\begin{eqnarray}
 \tilde M_{X,Y}(\xi_p,\xi_q) \!\!\! &= & \!\!\! \log [C(p,q)] 
 \nonumber \\
 &+ & \!\!\! \frac{1}{C(p, q)} \int_{0}^{p}  {\rm d}u \int_{0}^{q} {c(u, v)}\,
 \log\frac{c(u, v)}
 {\int_{0}^{q} c(u, w) \,{\rm d}w\, \int_{0}^{p}c(z, v)\,{\rm d}z}\,{\rm d}v.
 \label{eq:MpqXY}
\end{eqnarray}
\end{proposition}
\begin{proof}
Due to  \eqref{eq:ftildeX}, \eqref{eq:ftildeY} and \eqref{eq:jointftildeXY}, it follows that
the mutual information of  \eqref{eq:defXYcond} is
\begin{eqnarray*}
 \tilde M_{X,Y}(\xi_p,\xi_q) \!\!\! &=& \!\!\!  \int_0^{F_X^{-1}(p)}{\rm d}x\int_0^{F_Y^{-1}(q)}
 \frac{c(F_X(x), F_Y(y)) f_X(x)f_Y(y)}{C(p,q)}\,
 \\
& \times & \!\!\!  \log\frac{c(F_X(x), F_Y(y))\, C(p, q)}
 {\int_0^q c(F_X(x), v)\,{\rm d}v\,   \int_0^p c(u, F_Y(y))\,{\rm d}u}\,{\rm d}y,
 \qquad p,q\in (0,1).
\end{eqnarray*}
Finally, setting $u=F_X(x)$ and $v=F_Y(y)$, we obtain \eqref{eq:MpqXY}.
\end{proof}
\begin{example}\rm
Let $(X, Y )$ have copula
$$
 C(u,v)=\frac{u\,v}{u+v-u\,v}, \qquad u,v\in (0,1),
$$
i.e.\ a special case of a Clayton copula. From Proposition \ref{prop:copulapq}, it follows that the
mutual information of   \eqref{eq:defXYcond} is
$$
  \tilde M_{X,Y}(\xi_p,\xi_q)=-\frac{1}{2}+\log 2=0.1931,
  \qquad p,q\in(0,1).
$$
\end{example}
\par
Let us now consider the  joint survival function $\overline{F}(x,y)$ and the corresponding 
marginal survival functions $\overline{F}_X(x)=\mathbb P(X>x)$ and $\overline{F}_Y(y)=\mathbb P(Y>y)$. 
Similarly as in \eqref{eq:relFxyC}, these functions are related by
\begin{equation}
 \overline F(x, y) =\tilde C( \overline  F_X(x), \overline F_Y(y)), \qquad x,y\in\mathbb{R},
 \label{eq:FugtildeC}
\end{equation}
where $\tilde C(u,v)=1-u-v-C(u,v)$, $u,v\in(0,1)$, is the survival copula function.
The survival copula density, given by
$$
 \tilde c(u,v)=\frac{\partial^2}{\partial u\partial v } \tilde C(u, v)
 =\frac{\partial^2}{\partial u\partial v} \overline F(\overline F^{-1}_X(u), \overline F^{-1}_Y(v)), 
 \qquad u,v\in(0,1),
$$
allows us to express the joint density of $(X,Y)$ as
\begin{equation}
 f(x,y)= f_X(x) f_Y(y)\,\tilde c(\overline F_X(x), \overline F_Y(y)) ,
 \qquad x,y\in\mathbb{R}.
 \label{eq:relfxytilde}
\end{equation}
We recall that the copula density and the survival copula density are related by the following identity:
$c(u,v)=\tilde c(1-u,1-v)$, $u,v\in(0,1)$.
\par
In order to consider the residual mutual information we make use of \eqref{eq:FugtildeC} 
and \eqref{eq:relfxytilde} in  \eqref{equation:7}, and perform the substitution $v=F_Y(y)$ in 
the integral. Moreover, by setting $s$ and $t$ as in \eqref{eq:stpq}, the density of the marginal
residual lifetime $[X-F_X^{-1}(p)\,|\,X> F_X^{-1}(p), Y> F_Y^{-1}(q)]$ can be expressed as
\begin{equation}
 f_X(x;\xi_p,\xi_q) 
 =  \frac{f_X(x+F_X^{-1}(p))}{\tilde C(1-p, 1-q)}\,\int_q^{1} \tilde c(\overline F_X(x+F_X^{-1}(p)), 1-v) \,{\rm d}v 
 \qquad \hbox{for $x\geq 0$}. 
 \label{eq_fxstnew}
\end{equation}
Similarly, from \eqref{equation:8}, it follows that the density of
$[Y-F_Y^{-1}(q)\,|\,X> F_X^{-1}(p), Y> F_Y^{-1}(q)]$ is given by
\begin{equation}
 f_Y(y;\xi_p,\xi_q)
 = \frac{f_Y(y+F_Y^{-1}(q))}{\tilde C(1-p, 1-q)}\,\int_p^{1} \tilde c(1-u, \overline F_Y(y+F_Y^{-1}(q)))
 \,{\rm d}u \qquad \hbox{for $y\geq0$}. 
 \label{eq_fystnew}
\end{equation}
Furthermore,  the density of the joint residual lifetimes
\begin{equation}
 [(X-F_X^{-1}(p), Y-F_Y^{-1}(q))\,|\,X> F_X^{-1}(p), Y> F_Y^{-1}(q)],  \qquad p,q\in(0,1),
 \label{eq:defjreslif}
\end{equation}
is
\begin{eqnarray}
 f_{X,Y}(x,y;\xi_p,\xi_q) \!\!\! &=& \!\!\! f_X(x+F_X^{-1}(p)) f_Y(y+F_Y^{-1}(q))
 \nonumber \\
 &\times & \!\!\! \frac{\tilde c(\overline F_X(x+F_X^{-1}(p)), \overline F_Y(y+F_Y^{-1}(q)))}{\tilde C(1-p,1-q)}
 \qquad \hbox{for $x\geq0$ and $y\geq0$.}
 \label{eq_fxystnew}
\end{eqnarray}
In conclusion, we obtain the dynamic mutual information for residual lifetimes
in terms of survival copula.
\begin{proposition}
The mutual information of the bivariate residual lifetimes \eqref{eq:defjreslif} for all $p,q\in(0,1)$ is given by
\begin{eqnarray}
 M_{X,Y}(\xi_p,\xi_q) \!\!\! &=& \!\!\! \log[\tilde C(1-p,1-q)]
 + \frac{1}{\tilde C(1-p,1-q)}
 \nonumber \\
 &\times &  \!\!\! \int_0^{1-p} {\rm d}z \int_0^{1-q} \tilde c(z, w) \log\frac{\tilde c(z, w)}
 {\int_0^{1-q} \tilde c( z, v)\,{\rm d}v\,
 \int_0^{1-p} \tilde c(u, w)\,{\rm d}u}\,\,{\rm d}w.
 \label{eq:Mpqfinale}
\end{eqnarray}
\end{proposition}
\begin{proof}
Making use of Eqs.\ \eqref{eq_fxstnew}, \eqref{eq_fystnew} and \eqref{eq_fxystnew},
for $p,q\in(0,1)$, we can write
\begin{eqnarray*}
 M_{X,Y}(\xi_p,\xi_q)  \!\!\! &=&  \!\!\! \int_0^{+\infty} {\rm d}x \int_0^{+\infty}
 \frac{  \tilde c(\overline F_X(x+F_X^{-1}(p)), \overline F_Y(y+F_Y^{-1}(q)))}
 {\tilde C(1-p,1-q)}\,
 \\
 &\times &  \!\!\! f_X(x+F_X^{-1}(p)) f_Y(y+F_Y^{-1}(q)) \\
 & \times &  \!\!\! \log\frac{\tilde C(1-p,1-q) \, \tilde c(\overline F_X(x+F_X^{-1}(p)), \overline F_Y(y+F_Y^{-1}(q)))}
 {\int_q^{1} \tilde c(\overline F_X(x+F_X^{-1}(p)), 1-v)\,{\rm d}v\,
 \int_p^{1} \tilde c(1-u, \overline F_Y(y+F_Y^{-1}(q)))\,{\rm d}u}\,{\rm d}y.
\end{eqnarray*}
Hence, setting $z=\overline F_X(x+F_X^{-1}(p))$ and $w= \overline F_Y(y+F_Y^{-1}(q))$ we obtain
\eqref{eq:Mpqfinale}.
\end{proof}
%
\section*{Acknowledgements}
This work is partially supported by INdAM-GNCS, by FARO (Universit\`a di Napoli Federico II)
and by Regione Campania (Legge 5).

%
\end{document}